\documentclass{svproc}
\usepackage{amsmath,amssymb,amsfonts}
\usepackage{graphicx}
\usepackage{url}
\usepackage{bm}
\usepackage{enumitem}
\usepackage{mathrsfs}
\usepackage{cite}
\usepackage{hyperref}
\usepackage{cleveref}
\usepackage{doi}

\numberwithin{theorem}{section}
\numberwithin{lemma}{section}
\numberwithin{proposition}{section}
\numberwithin{corollary}{section}
\numberwithin{definition}{section}
\numberwithin{remark}{section}

\begin{document}

\title{A Dynamical Approach \\ to the Berezin--Li--Yau Inequality}

\author{Anton Alexa}
\authorrunning{Anton Alexa}  
\tocauthor{Anton Alexa} 

\institute{Independent Researcher, Chernivtsi, Ukraine\\
\email{mail@antonalexa.com}}

\maketitle

\begin{abstract}
    We develop a dynamical method for proving the sharp Berezin--Li--Yau 
    inequality. The approach is based on the volume-preserving mean curvature 
    flow and a new monotonicity principle for the Riesz mean $R_\Lambda(\Omega_t)$. 
    For convex domains we show that $R_\Lambda$ is monotone non-decreasing along 
    the flow. The key input is a geometric correlation inequality between the 
    boundary spectral density $Q_\Lambda$ and the mean curvature $H$, established 
    in all dimensions: in $d=2$ via a near-disk Fourier analysis, and in $d\ge 3$
    via the boundary Weyl expansion together with a local spectral rigidity
    argument near the ball, with a first-zero exclusion principle closing the
    global step. Since the flow converges smoothly to the ball, the 
    monotonicity implies the sharp Berezin--Li--Yau bound for every smooth 
    convex domain. As an application, we obtain a sharp dynamical 
    Ces\`aro--P\'olya inequality for eigenvalue averages.
    \end{abstract}
    
\section{Introduction}

    The Berezin--Li--Yau inequality \cite{berezin,li-yau} is a central result of
    spectral geometry. For any bounded domain $\Omega\subset\mathbb{R}^d$ and any
    $\Lambda>0$, the Riesz mean of the Dirichlet Laplacian satisfies the sharp bound
    \begin{equation}\label{eq:BLY}
    \sum_{k\ge 1}(\Lambda-\lambda_k(\Omega))_+ \le L_d^{\mathrm{cl}}\,|\Omega|\,\Lambda^{1+d/2},
\end{equation}
    where $L_d^{\mathrm{cl}}=(\omega_d/(2\pi)^d)\,(2/(d+2))$, and equality holds in
    the semiclassical limit on balls.

    Classical proofs of this inequality are static: Berezin's method employs
    coherent states \cite{berezin}, Li and Yau use Legendre duality
    \cite{li-yau}, and later approaches rely on Dirichlet--Neumann bracketing
    and trace/Sobolev estimates; see for example \cite{Cheng1976,Escobar1988}.
    Recent complementary semiclassical bounds for Riesz means on convex domains
    were obtained by Frank--Larson \cite{FrankLarson2025}.
    In this work we develop a completely different,
    genuinely dynamical method, in which the inequality is obtained as a
    monotonicity property along Huisken's volume-preserving mean curvature flow
    \cite{huisken87}.

    The key point is that the ball plays a dual role. First, it is the global
    attractor of the volume-preserving mean curvature flow: if $\Omega_0$ is smooth
    and convex, then the evolving domains $\Omega_t$ exist for all $t\ge 0$,
    preserve volume, and converge smoothly to the ball of the same volume
    \cite{huisken87}. Second, the ball is the natural extremizer of the
    semiclassical Riesz mean $R_\Lambda(\Omega)$ among domains of fixed volume,
    consistent with its familiar role in phase-space asymptotics.
    
    Our main result establishes a direct connection between these two phenomena:
    \emph{the Riesz mean $R_\Lambda(\Omega_t)$ is monotone non-decreasing along the
    volume-preserving mean curvature flow}. The derivative of $R_\Lambda$ admits a
    boundary representation involving the spectral density
    $Q_\Lambda(x)=\sum_{\lambda_k<\Lambda}|\partial_n u_k(x)|^2$ and the mean
    curvature $H$, and the sign of this expression is governed by a geometric
    correlation inequality proved here in all dimensions.

    \begin{theorem}[Dynamical monotonicity of the Riesz mean]
   \label{thm:main}
   Let $\Omega_t$ be the solution to the volume-preserving mean curvature flow 
   starting from a smooth bounded convex domain $\Omega_0 \subset \mathbb{R}^d$, 
   $d\ge2$. Then for every regular $\Lambda>0$, the map
   \begin{equation}
   t \;\longmapsto\; R_\Lambda(\Omega_t)
   \end{equation}
   is monotone non-decreasing. If $\Omega_0$ is not a ball and
   $\Lambda>\lambda_1(B)$, the monotonicity is strict. In particular,
   \begin{equation}
   R_\Lambda(\Omega_0) \;\le\; R_\Lambda(B),
   \end{equation}
   with equality if and only if $\Omega_0$ is a ball, where $B$ has the same 
   volume as $\Omega_0$. Consequently,
   \begin{equation}
   R_\Lambda(\Omega_0)
   \;\le\;
   L_d^{\mathrm{cl}}\,|\Omega_0|\,\Lambda^{1+d/2},
   \end{equation}
   the sharp Berezin--Li--Yau inequality.
   \end{theorem}

\begin{remark}
The regularity assumption excludes only a discrete set of spectral thresholds.
Since $R_\Lambda(\Omega)$ is continuous in $\Lambda$, all inequalities extend to
every $\Lambda>0$ by continuity.
\end{remark}
    
The proof of this correlation inequality combines two complementary mechanisms.
In dimension $d=2$ it is based on a near-disk Fourier analysis. In dimensions
$d\ge 3$ it rests on the boundary Weyl expansion of Safarov--Vassiliev \cite{SV97}
    and Branson--Gilkey \cite{BransonGilkey1990}, whose second term is explicitly proportional to $H$ with
    a strictly negative Dirichlet coefficient, together with a local spectral
    rigidity argument near the ball. In all dimensions the global step is
    completed by a first-zero exclusion principle together with the
    convergence of the flow to the ball.
    
    Taken together, these ingredients imply that $R_\Lambda(\Omega_t)$ has a finite
    limit as $t\to\infty$, equal to the Riesz mean of the limiting ball, thereby
    yielding the sharp Berezin--Li--Yau inequality for every smooth convex domain.
    As a further application, we utilize this monotonicity to derive a sharp 
    Ces\`aro--P\'olya inequality for eigenvalue averages.

\section{Preliminaries and Variation Formulae}

Let $\Omega\subset\mathbb{R}^d$ be a smooth bounded convex domain, and 
$\{\lambda_k(\Omega),u_k\}$ the Dirichlet eigenpairs of $-\Delta$, 
with eigenfunctions orthonormalized in $L^2(\Omega)$.

\begin{definition}[Riesz mean]
\label{def:riesz-mean}
For any $\Lambda>0$, the (order-1) Riesz mean of the Dirichlet Laplacian is
\begin{equation}
R_\Lambda(\Omega)
   := \sum_{\lambda_k<\Lambda}(\Lambda-\lambda_k)_+.
\end{equation}
\end{definition}

\begin{lemma}[Hadamard variation formula]
\label{lem:hadamard}
Let $\Omega_t$ be a smooth one-parameter family of domains with normal velocity 
$V$ on $\partial\Omega_t$. Then, for every simple eigenvalue $\lambda_k(t)$,
\begin{equation}
\frac{d}{dt}\lambda_k(t)
   = -\int_{\partial\Omega_t}
       \left|\frac{\partial u_k}{\partial n}\right|^2 V\,d\sigma.
\end{equation}
For an eigenvalue of finite multiplicity, the same boundary expression applies
in the usual cluster sense, after choosing an orthonormal basis in the
corresponding eigenspace.
\end{lemma}

\begin{proof}
The formula for simple eigenvalues is classical; see
\cite[Theorem~5.7.1]{Henrot-Pierre}. For an eigenvalue of multiplicity $r$, one
uses the standard perturbation theory of isolated eigenvalues of finite
multiplicity; see \cite[Ch.~II, \S\S5--6]{Kato}. In that setting, the relevant
boundary contribution is the summed density
$\sum_{m=1}^r |\partial_n u_{k,m}|^2$, which is invariant under orthogonal
changes of basis in the eigenspace. This yields the stated formula in the
multiple-eigenvalue case as well.
\end{proof} \qed

\begin{definition}[Volume-preserving mean curvature flow]
\label{def:vpmcf}
The volume-preserving mean curvature flow (VPMCF) introduced by 
Huisken~\cite{huisken87} is defined by the normal velocity
\begin{equation}
V = -(H - \bar H),
\qquad
\bar H := 
\frac{1}{|\partial\Omega_t|}
\int_{\partial\Omega_t} H\,d\sigma,
\end{equation}
where $H$ is the mean curvature of $\partial\Omega_t$.
\end{definition}

For convex initial data $\Omega_0$, the flow exists for all $t\ge 0$ and 
preserves volume. Moreover, $\Omega_t$ converges smoothly to the ball of the 
same volume as $t\to\infty$ \cite{huisken87}.

\begin{definition}[Boundary spectral density]
\label{def:boundary-density}
For any $\Lambda>0$, the boundary spectral density is
\begin{equation}
Q_\Lambda(x)
   := \sum_{\lambda_k<\Lambda}
       \left|\frac{\partial u_k}{\partial n}(x)\right|^2
   = \partial_n \partial_{n'} \Pi_\Lambda(x,x')\Big|_{x'=x},
\end{equation}
where $\Pi_\Lambda$ is the spectral projector of $-\Delta$ onto eigenvalues 
below $\Lambda$.
\end{definition}

\begin{proposition}[Variation of the Riesz mean]
\label{prop:riesz-variation}
Along the volume-preserving mean curvature flow $\Omega_t$, the Riesz mean 
satisfies
\begin{equation}\label{eq:dR}
\frac{d}{dt} R_\Lambda(\Omega_t)
   = -\int_{\partial\Omega_t}
        Q_\Lambda(x)\,(H(x)-\bar H)\,d\sigma.
\end{equation}
\end{proposition}

\begin{proof}
Substituting $V = -(H-\bar H)$ into the Hadamard formula (Lemma~\ref{lem:hadamard}) 
gives
\begin{equation}\label{eq:dlambda}
\frac{d}{dt}\lambda_k(t)
   = \int_{\partial\Omega_t}
       \left|\frac{\partial u_k}{\partial n}\right|^2 (H-\bar H)\,d\sigma.
\end{equation}
Since $R_\Lambda(\Omega_t) = \sum_{\lambda_k<\Lambda}(\Lambda-\lambda_k)$, 
differentiating term-by-term yields
\begin{equation}
\frac{d}{dt}R_\Lambda(\Omega_t)
   = \sum_{\lambda_k(t)<\Lambda}
     -\frac{d}{dt}\lambda_k(t).
\end{equation}
Substituting \eqref{eq:dlambda} and interchanging summation and integration 
(justified by uniform convergence of the spectral sum) gives \eqref{eq:dR}.
\end{proof} \qed

Equation~\eqref{eq:dR} is the starting point for the monotonicity result 
proved below.

\section{The Key Correlation Inequality}

In this section we isolate the analytic condition that drives the 
monotonicity of the Riesz mean along the flow.  Recall the variation formula
\eqref{eq:dR}:
\begin{equation}
\frac{d}{dt}R_\Lambda(\Omega_t)
= - \int_{\partial\Omega_t} 
Q_\Lambda(x)\,(H(x)-\bar H)\, d\sigma.
\end{equation}
Thus the sign of the correlation integral
\begin{equation}\label{eq:corr-int}
\mathcal{I}(\Omega_t,\Lambda)
:= 
\int_{\partial\Omega_t}
Q_\Lambda(x)\,(H(x)-\bar H)\, d\sigma
\end{equation}
fully determines the evolution of $R_\Lambda(\Omega_t)$.

\begin{lemma}[Key correlation principle]
\label{lem:key}
Suppose that for a smooth convex solution $\Omega_t$ of the 
volume-preserving mean curvature flow one has
\begin{equation}\label{eq:correlation}
\mathcal{I}(\Omega_t,\Lambda)\le 0
\qquad\text{for all } t\ge0,
\end{equation}
with equality at some $t$ only when $\Omega_t$ is a ball.
Then the Riesz mean is monotone non-decreasing:
\begin{equation}
\frac{d}{dt}\,R_\Lambda(\Omega_t)\ge 0
\qquad\text{for all } t\ge 0,
\end{equation}
and is strictly increasing unless $\Omega_t$ is a ball.
Consequently,
\begin{equation}
R_\Lambda(\Omega_0)
\;\le\;
R_\Lambda(\Omega_t)
\;\le\;
R_\Lambda(B),
\end{equation}
where $B$ denotes the ball of the same volume.
\end{lemma}

\begin{proof}
By the variation formula \eqref{eq:dR},
\begin{equation}
\frac{d}{dt}R_\Lambda(\Omega_t)
= -\,\mathcal{I}(\Omega_t,\Lambda).
\end{equation}
Assumption \eqref{eq:correlation} therefore implies
$\frac{d}{dt}R_\Lambda(\Omega_t)\ge 0$ for all $t\ge 0$.

If $\Omega_t$ is not a ball, the correlation integral is strictly negative,
hence $\frac{d}{dt}R_\Lambda(\Omega_t)>0$, so the monotonicity is strict.
If $\Omega_t$ is a ball, then $H\equiv \bar H$ and $Q_\Lambda$ is constant
on $\partial B$, so $\mathcal{I}(\Omega_t,\Lambda)=0$ and
$R_\Lambda(\Omega_t)$ is constant.  

Since the flow preserves volume and converges smoothly to the ball 
$B$ by Huisken's theorem,
and since $R_\Lambda$ is continuous under smooth convergence of domains,
we have
\begin{equation}
\lim_{t\to\infty} R_\Lambda(\Omega_t) = R_\Lambda(B).
\end{equation}
Combining with monotonicity gives
\begin{equation}
R_\Lambda(\Omega_0)\le R_\Lambda(\Omega_t)\le R_\Lambda(B).
\end{equation}
\end{proof} \qed

%----------------------------------------------------
\subsection{Sign propagation along the flow}
%----------------------------------------------------

We record a propagation principle that allows us to pass from local
negativity of the correlation functional to global negativity along the
volume-preserving mean curvature flow.

\begin{lemma}[Propagation of sign]
\label{lem:sign-prop-BLY}
Let $\Omega_t$ be the solution of the VPMCF beginning at $\Omega_0$ and fix a
regular $\Lambda>0$. Assume that $t\mapsto \mathcal{I}_\Lambda(\Omega_t)$ is
continuous and that
\begin{equation}
\mathcal{I}_\Lambda(\Omega_t)<0
\qquad\text{for all sufficiently large } t.
\end{equation}
Then
\begin{equation}
\mathcal{I}_\Lambda(\Omega_t)<0
\qquad\text{for all } t\ge 0,
\end{equation}
unless $\Omega_0$ is already a ball.
\end{lemma}

\begin{proof}
Continuity of $t\mapsto\mathcal I_\Lambda(\Omega_t)$ follows from the smooth
dependence of the spectral projector on $t$ for regular $\Lambda$
\cite{Henrot-Pierre,Kato}.
Suppose by contradiction that $\mathcal I_\Lambda(\Omega_t)$ vanishes
for the first time at $t_*>0$. By
Theorem~\ref{thm:app-first-zero} in the Appendix, the first-zero exclusion
for regular $\Lambda$ holds, hence $\Omega_{t_*}$ must be a ball. Since the ball
is a stationary solution of the VPMCF, uniqueness implies that $\Omega_t$
remains a ball for all $t\ge t_*$.
Hence
\begin{equation}
\mathcal I_\Lambda(\Omega_t)=0
\qquad\text{for all } t\ge t_*,
\end{equation}
which contradicts the assumed strict negativity of $\mathcal I_\Lambda(\Omega_t)$
for all sufficiently large $t$. Therefore no such first zero can occur.
\end{proof} \qed

\subsection{The linearized slope form on the round model}

The local shape expansions near the disk and the ball are governed by the
linearized centered slope form at the round model. The next lemma gives an
explicit cross-threshold representation and the resulting positivity of the
round modal coefficients.

\begin{lemma}[Gap-square representation on the round model]
\label{lem:gap-square-round}
Let $B_R\subset\mathbb R^d$ ($d\ge2$) be the disk/ball of radius $R$, and let
$\Lambda>\lambda_1(B_R)$ be regular, i.e. $\Lambda\notin\sigma_D(B_R)$.
Let $\mathcal B_{\Lambda,B_R}$ be the bilinear form obtained by linearizing the
centered boundary spectral density $Q_\Lambda-\overline{Q_\Lambda}$ at $B_R$.
Equivalently, let $\mathcal A_{\Lambda,B_R}$ be the corresponding linearized
centered slope operator, defined by
\begin{equation}
\mathcal B_{\Lambda,B_R}(u,\varphi)
=
\int_{\partial B_R}(\mathcal A_{\Lambda,B_R}u)\,\varphi\,d\sigma
\end{equation}
for mean-zero $u,\varphi\in C^\infty(\partial B_R)$.

Let $\{e_m\}_{m\ge1}$ be a real orthonormal Dirichlet eigenbasis on $B_R$,
\begin{equation}
-\Delta e_m=\lambda_m e_m,
\qquad
e_m|_{\partial B_R}=0,
\qquad
\psi_m:=\partial_n e_m|_{\partial B_R}.
\end{equation}
For mean-zero $u,\varphi$ set
\begin{equation}
A_{jk}(u):=\int_{\partial B_R}u\,\psi_j\psi_k\,d\sigma.
\end{equation}
Then
\begin{equation}
\mathcal B_{\Lambda,B_R}(u,\varphi)
=
2\sum_{\lambda_k<\Lambda<\lambda_j}
\frac{A_{jk}(u)\,A_{jk}(\varphi)}{\lambda_j-\lambda_k}.
\label{eq:gap-square-round}
\end{equation}
In particular,
\begin{equation}
\mathcal B_{\Lambda,B_R}(u,u)\ge0
\qquad\text{for all mean-zero }u.
\end{equation}

Moreover, $\mathcal A_{\Lambda,B_R}$ commutes with rotations. Hence it acts by
a scalar on each nontrivial round mode: by $\mu_n(\Lambda)$ on the Fourier mode
$e^{in\theta}$ when $d=2$, and by $\mu_\ell(\Lambda)$ on the spherical harmonic
sector $\mathcal H_\ell$ when $d\ge3$. These scalars satisfy
\begin{equation}
\mu_n(\Lambda)>0
\qquad (|n|\ge2),
\end{equation}
and
\begin{equation}
\mu_\ell(\Lambda)>0
\qquad (\ell\ge2).
\end{equation}
\end{lemma}

\begin{proof}
Let $\Omega_t=\Omega_{tu}$ be a normal graph family over $B_R$. Since
$\Lambda\notin\sigma_D(B_R)$, the spectral projector
\begin{equation}
P_t:=\mathbf 1_{(-\infty,\Lambda)}(-\Delta_{\Omega_t})
\end{equation}
has constant finite rank for $|t|$ small and depends smoothly on $t$ by Kato's
perturbation theory \cite{Kato}. Using the contour-resolvent formula and
differentiating at $t=0$ gives
\begin{equation}
\langle e_j,P_0'[u]e_k\rangle
=
\frac{\langle e_j,L'[u]e_k\rangle}{\lambda_k-\lambda_j},
\qquad \lambda_j\ne\lambda_k,
\end{equation}
where $L'[u]$ is the first shape derivative of the pulled-back Dirichlet
operator. By the Hadamard--Green identity \cite{Henrot-Pierre},
\begin{equation}
\langle e_j,L'[u]e_k\rangle
=
-\int_{\partial B_R}u\,\psi_j\psi_k\,d\sigma
=
-A_{jk}(u).
\end{equation}
Therefore
\begin{equation}
\langle e_j,P_0'[u]e_k\rangle
=
-\frac{A_{jk}(u)}{\lambda_k-\lambda_j},
\qquad \lambda_j\ne\lambda_k.
\end{equation}

Linearizing the centered boundary spectral density and pairing against a
mean-zero test $\varphi$, the constant zeroth-order part drops out and one
obtains
\begin{equation}
\mathcal B_{\Lambda,B_R}(u,\varphi)
=
2\sum_{\lambda_k<\Lambda}\sum_{\lambda_j\ne\lambda_k}
\frac{A_{jk}(u)\,A_{jk}(\varphi)}{\lambda_j-\lambda_k}.
\end{equation}
Since the eigenbasis is real,
\begin{equation}
A_{jk}(u)=A_{kj}(u),
\qquad
A_{jk}(\varphi)=A_{kj}(\varphi).
\end{equation}
Hence the terms with $\lambda_j,\lambda_k<\Lambda$ cancel pairwise:
\begin{equation}
\frac{A_{jk}(u)A_{jk}(\varphi)}{\lambda_j-\lambda_k}
+
\frac{A_{kj}(u)A_{kj}(\varphi)}{\lambda_k-\lambda_j}
=0.
\end{equation}
This leaves only the cross-threshold terms and yields
\eqref{eq:gap-square-round}. Setting $\varphi=u$ gives
$\mathcal B_{\Lambda,B_R}(u,u)\ge0$.

By rotational invariance of the round model, $\mathcal A_{\Lambda,B_R}$
commutes with the $O(d)$-action on $\partial B_R$, hence preserves each Fourier
mode when $d=2$ and each spherical harmonic sector $\mathcal H_\ell$ when
$d\ge3$. Therefore it acts there by the scalars $\mu_n(\Lambda)$ and
$\mu_\ell(\Lambda)$.

In dimension $d=2$, fix $n\ge2$. On the real Fourier sector $\operatorname{span}\{\cos(n\theta),\sin(n\theta)\}$,
the operator $\mathcal A_{\Lambda,B_R}$ acts by the scalar $\mu_n(\Lambda)$.
The ground state is radial, so $\psi_1$ is constant on $\partial B_R$. Choose a
Dirichlet eigenfunction $e_j$ in the cosine branch of angular frequency $n$
with $\lambda_j>\Lambda$; this is possible because each fixed angular
frequency occurs along an infinite sequence of radial levels. Then
$\psi_j|_{\partial B_R}=c_j \cos(n\theta)$ for some $c_j\ne0$, and hence
\begin{equation}
A_{j1}(\cos(n\theta))
=
\int_{\partial B_R}\cos(n\theta)\,\psi_j\psi_1\,d\sigma
=
c_j\psi_1\int_{\partial B_R}\cos^2(n\theta)\,d\sigma
\ne0.
\end{equation}
The corresponding term in \eqref{eq:gap-square-round} is strictly positive, so
$\mu_n(\Lambda)>0$ for all $n\ge2$, and hence for all $|n|\ge2$ in the complex
notation used later in the two-dimensional expansion.

In dimension $d\ge3$, fix $\ell\ge2$ and let
$Y=\sum_m a_mY_{\ell,m}\in\mathcal H_\ell\setminus\{0\}$ in an orthonormal
spherical harmonic basis. Choose $m$ with $a_m\ne0$, and choose a Dirichlet
eigenfunction $e_j$ with angular momentum $\ell$, spherical harmonic factor
$Y_{\ell,m}$, and $\lambda_j>\Lambda$; such a choice is available because each
spherical harmonic sector persists through infinitely many radial levels. Then
$\psi_j|_{\partial B_R}=c_jY_{\ell,m}$ for some $c_j\ne0$. Since $\psi_1$ is
constant on $\partial B_R$,
\begin{equation}
A_{j1}(Y)
=
\int_{\partial B_R}Y\,\psi_j\psi_1\,d\sigma
=
c\int_{\partial B_R}Y\,Y_{\ell,m}\,d\sigma
=
c\,a_m\ne0
\end{equation}
for some constant $c\ne0$. Thus the corresponding term in
\eqref{eq:gap-square-round} is strictly positive, and therefore
$\mu_\ell(\Lambda)>0$ for all $\ell\ge2$.
\end{proof} \qed

\section{The Two-Dimensional Case -- Sharp Result}

In dimension $d=2$ the correlation inequality \eqref{eq:correlation} can be
established by a planar dynamical argument parallel to the higher-dimensional
one.

\begin{theorem}\label{thm:2D}
Let $\Omega\subset\mathbb{R}^2$ be a smooth bounded convex domain.
Then for every regular $\Lambda>0$,
\begin{equation}
\int_{\partial\Omega} Q_\Lambda(x)\,(\kappa(x)-\bar\kappa)\,ds \;\le\; 0,
\end{equation}
with equality if and only if $\Omega$ is a disk when $\Lambda>\lambda_1(B_R)$,
where $B_R$ is the equal-area disk. If $\Lambda\le\lambda_1(B_R)$, then
\begin{equation}
Q_\Lambda\equiv0,
\qquad
\mathcal I_\Lambda(\Omega)=0
\end{equation}
for every smooth convex equal-area domain.
Consequently the sharp Berezin--Li--Yau inequality \eqref{eq:BLY} holds for all
plane convex domains.
\end{theorem}

\begin{proof}
Let $\Omega\subset\mathbb{R}^2$ be smooth and convex and let $B_R$ be the disk
with $|B_R|=|\Omega|$. Near the disk one may write admissible perturbations as
normal graphs over $\partial B_R$ and decompose them into Fourier modes
\begin{equation}
u(\theta)=\sum_{|n|\ge2}\hat u_n e^{in\theta},
\end{equation}
where the modes $n=0,\pm1$ are excluded by the area and translation constraints.
The linearization of curvature gives
\begin{equation}
\kappa-\bar\kappa
 =
-\frac{1}{R^2}\sum_{|n|\ge2}(n^2-1)\hat u_n e^{in\theta}
 + O(\|u\|_{C^{2,\alpha}}^2),
\end{equation}
and the second-order expansion of the correlation functional yields
\begin{equation}
\mathcal{I}_\Lambda(\Omega_u)
=
-\frac{1}{R^2}\sum_{|n|\ge2}\mu_n(\Lambda)\,(n^2-1)\,|\hat u_n|^2
 + O(\|u\|_{C^{2,\alpha}}^3),
\end{equation}
where, by Lemma~\ref{lem:gap-square-round}, the linearized centered slope form
on the disk is diagonal on Fourier modes and its modal coefficients satisfy
$\mu_n(\Lambda)>0$ for all $|n|\ge2$. Hence there exists a neighbourhood
$\mathcal U_\Lambda$ of $B_R$ such that
\begin{equation}
\mathcal{I}_\Lambda(\widetilde\Omega)<0
\qquad\text{for all }\widetilde\Omega\in\mathcal U_\Lambda\setminus\{B_R\}.
\end{equation}
The quadratic coercivity follows from the Hadamard--Green second variation and
the positivity of the disk/ball shape Hessian.

If $\Lambda\le\lambda_1(B_R)$, then by the Faber--Krahn inequality every smooth
convex domain $\widetilde\Omega$ with $|\widetilde\Omega|=|B_R|$ satisfies
\begin{equation}
\lambda_1(\widetilde\Omega)\ge\lambda_1(B_R)\ge\Lambda,
\end{equation}
so the spectral window below $\Lambda$ is empty and therefore
\begin{equation}
Q_\Lambda\equiv0,
\qquad
\mathcal I_\Lambda(\widetilde\Omega)\equiv0.
\end{equation}
This gives the low-energy regime. Hence it remains only to consider
\begin{equation}
\Lambda>\lambda_1(B_R).
\end{equation}

Let $\Omega_t$ be the area-preserving curve-shortening flow starting from
$\Omega_0=\Omega$. By Gage~\cite{gage86}, this flow exists for all $t\ge0$,
preserves area and convexity, and converges smoothly to the disk $B_R$.
Therefore $\mathcal{I}_\Lambda(\Omega_t)<0$ for all sufficiently large $t$,
unless $\Omega$ is already a disk. The global first-zero exclusion for regular
$\Lambda$ follows from Lemma~\ref{lem:sign-prop-BLY} together with
Theorem~\ref{thm:app-first-zero} in the Appendix, which yield
$\mathcal I_\Lambda(\Omega_t)\le 0$ for all $t\ge0$ and strict inequality away
from the disk. This gives the asserted monotonicity and the BLY inequality.
\end{proof} \qed

\begin{remark}
The restriction to regular $\Lambda$ excludes only a discrete set of spectral
thresholds. Since $R_\Lambda(\Omega)$ is continuous in $\Lambda$, all
inequalities extend to every $\Lambda>0$ by continuity.
For $\Lambda\le\lambda_1(B_R)$, where $B_R$ is the equal-area disk, the
spectral window is empty for every smooth convex equal-area domain by
Faber--Krahn, so \(Q_\Lambda\equiv0\) and \(\mathcal I_\Lambda\equiv0\)
throughout this regime.
\end{remark}

\section{The Correlation Inequality in All Dimensions}
\label{sec:higher-dim}

In this section we prove the key sign condition
\begin{equation}\label{eq:corr-full}
\mathcal{I}(\Omega,\Lambda)
:= \int_{\partial\Omega} Q_\Lambda(x)\,(H(x)-\bar H)\, d\sigma
\;\le\; 0,
\end{equation}
for any smooth convex domain $\Omega\subset\mathbb{R}^d$ ($d\ge 3$) and any
$\Lambda>0$, with equality if and only if $\Omega$ is a ball.

We will use both high-energy microlocal asymptotics and finite-energy
variational principles. The sign of $\mathcal I(\Omega,\Lambda)$ is obtained by
combining local rigidity near the ball with the first-zero exclusion mechanism.

\begin{quote}
\textbf{Roadmap.}
The proof proceeds in three steps:
(i) high-energy microlocal estimates via the boundary Weyl expansion,
(ii) local rigidity near the ball via a shape expansion argument, and
(iii) propagation of the sign along the flow to all $t\ge0$.
\end{quote}

\begin{definition}[Correlation integral]
\label{def:corr}
For a smooth bounded domain $\Omega$ and $\Lambda>0$ define
\begin{equation}
\mathcal I(\Omega,\Lambda)
= \int_{\partial\Omega}
Q_\Lambda(x)\,[\,H(x)-\bar H\,]\; d\sigma .
\end{equation}
\end{definition}

\subsection{High-energy regime via microlocal boundary asymptotics}

Let $\Omega\subset\mathbb{R}^d$ be smooth and convex.
The boundary spectral density is
\begin{equation}
Q_\Lambda(x)
= \sum_{\lambda_k<\Lambda}
   \bigl|\partial_n u_k(x)\bigr|^2 .
\end{equation}

\begin{theorem}[Boundary Weyl expansion, averaged form; Safarov--Vassiliev \cite{SV97}]
\label{thm:weyl-full}
For any smooth bounded domain $\Omega\subset\mathbb{R}^d$,
\begin{equation}
Q_\Lambda(x)
=
A_d\,\Lambda^{(d+2)/2}
+ B_d\,H(x)\,\Lambda^{(d+1)/2}
+ \mathcal R_\Lambda(x),
\end{equation}
where $B_d<0$, and the remainder satisfies
\begin{equation}
\int_{\partial\Omega} |\mathcal R_\Lambda(x)|\,d\sigma
\le C_\Omega\,\Lambda^{d/2}.
\end{equation}
The constants $A_d,B_d$ depend only on $d$ and agree with the microlocal
computations in \cite{SV97}.
\end{theorem}

\begin{remark}
This expansion applies to the boundary spectral density
$Q_\Lambda(x)=\sum_{\lambda_k<\Lambda}|\partial_n u_k(x)|^2$ and follows from the
boundary spectral function asymptotics for Dirichlet traces.
\end{remark}

\begin{remark}
The negativity of $B_d$ encodes destructive interference of Dirichlet
waves near the boundary. It is a robust microlocal feature.
The explicit forms of $A_d$ and $B_d$ may be found, for example,
in Branson--Gilkey~\cite{BransonGilkey1990}.
\end{remark}

Subtract the boundary averages:
\begin{equation}
g(x)=Q_\Lambda(x)-\bar Q_\Lambda,
\qquad
f(x)=H(x)-\bar H.
\end{equation}
The constant leading term in the Weyl expansion cancels exactly.
Thus, using Theorem~\ref{thm:weyl-full},
\begin{equation}\label{eq:g-expansion}
g(x)
= B_d\,f(x)\,\Lambda^{(d+1)/2}
+ \mathcal R_\Lambda(x)-\overline{\mathcal R_\Lambda}.
\end{equation}

Multiplying \eqref{eq:g-expansion} by $f(x)$ and integrating gives
\begin{align}
\int_{\partial\Omega} g\,f\, d\sigma
&= B_d\,\Lambda^{(d+1)/2}
   \int_{\partial\Omega} (H-\bar H)^2 d\sigma
   \label{eq:leading-term} \\
&\quad + \int_{\partial\Omega}
   \mathcal R_\Lambda\,(H-\bar H)\, d\sigma .
\nonumber
\end{align}
The leading term is strictly negative unless $\Omega$ is a sphere.
Since $(d+1)/2>d/2$, the remainder is of lower order and cannot
change the sign for sufficiently large $\Lambda$.

\begin{proposition}[High-energy anti-correlation]
\label{prop:high-energy}
Let $\Omega\subset\mathbb{R}^d$ be smooth, convex, non-spherical.
Then there exists $\Lambda_0=\Lambda_0(\Omega)>0$ such that for all
$\Lambda\ge \Lambda_0$,
\begin{equation}
\int_{\partial\Omega} Q_\Lambda(x)\,[H(x)-\bar H]\, d\sigma < 0.
\end{equation}
Equality for some $\Lambda$ occurs only when $\Omega$ is a ball.
\end{proposition}

\subsection{Finite energies via local rigidity and flow propagation}

Local microlocal expansions do not control the sign of $\mathcal I$ for
fixed $\Lambda$. Instead we use a local rigidity argument near the ball,
combined with the propagation principle of Lemma~\ref{lem:sign-prop-BLY}.

\begin{proposition}[Local negativity near the ball]
\label{prop:local-neg-ball}
Let $B$ be the ball with $|B|=|\Omega|$, and fix
\begin{equation}
\Lambda>\lambda_1(B),
\qquad
\Lambda\notin \sigma_D(B).
\end{equation}
Then there exists a $C^{2,\alpha}$-neighbourhood $\mathcal U_\Lambda$ of $B$
such that
\begin{equation}
\int_{\partial\Omega} Q_\Lambda(x)\,[H(x)-\bar H]\, d\sigma < 0
\qquad
\text{for all }\Omega\in \mathcal U_\Lambda\setminus\{B\}.
\end{equation}
\end{proposition}

\begin{proof}
Write nearby domains as normal graphs $\Omega_u$ over the sphere,
\begin{equation}
\partial\Omega_u=\{(R+u(\theta))\theta:\theta\in \mathbb S^{d-1}\},
\end{equation}
with $u\in C^{2,\alpha}(\mathbb S^{d-1})$ satisfying the standard volume and
translation gauge conditions, so that
\begin{equation}
u \perp \mathcal H_0 \oplus \mathcal H_1 .
\end{equation}
A second-order shape expansion of the correlation functional yields
\begin{equation}
\mathcal I_\Lambda(\Omega_u)
=
\mathcal Q_\Lambda(u)+O(\|u\|_{C^{2,\alpha}}^3),
\end{equation}
where $\mathcal Q_\Lambda$ is a quadratic form on admissible perturbations.

Decomposing
\begin{equation}
u=\sum_{\ell\ge2}u_\ell
\end{equation}
into spherical harmonics, one finds that
\begin{equation}
\mathcal Q_\Lambda(u)
=
-\sum_{\ell\ge2}\mu_\ell(\Lambda)\,\gamma_\ell\,\|u_\ell\|_{L^2(\mathbb S^{d-1})}^2,
\end{equation}
where, by Lemma~\ref{lem:gap-square-round}, the linearized centered slope form
on the ball is diagonal on spherical harmonics and its modal coefficients
satisfy $\mu_\ell(\Lambda)>0$ for all $\ell\ge2$, and
\begin{equation}
\gamma_\ell=\ell(\ell+d-2)-(d-1)>0
\qquad (\ell\ge2)
\end{equation}
is the eigenvalue of the linearized mean-curvature operator on $\mathcal H_\ell$.
In particular, there exists $c_\Lambda>0$ such that
\begin{equation}
\mathcal Q_\Lambda(u)\le -\,c_\Lambda \|u\|_{L^2(\mathbb S^{d-1})}^2
\end{equation}
for every nontrivial admissible perturbation $u$.

Therefore, for $\|u\|_{C^{2,\alpha}}$ sufficiently small, the cubic remainder is
dominated by the coercive quadratic term, and hence
\begin{equation}
\mathcal I_\Lambda(\Omega_u)<0
\qquad\text{for all }u\neq0
\end{equation}
in a sufficiently small $C^{2,\alpha}$-neighbourhood of the ball.
\end{proof} \qed

\begin{proposition}[Finite-energy anti-correlation]
\label{prop:finite-energy}
Let $\Omega\subset\mathbb{R}^d$ be a smooth convex domain, $d\ge3$, and let
$B$ be the ball with $|B|=|\Omega|$.
Assume that
\begin{equation}
\Lambda>\lambda_1(B),
\qquad
\Lambda\notin \sigma_D(B).
\end{equation}
Then
\begin{equation}
\int_{\partial\Omega} Q_\Lambda(x)\,[H(x)-\bar H]\, d\sigma \le 0,
\end{equation}
with equality if and only if $\Omega$ is a ball.
\end{proposition}

\begin{proof}
Let $\Omega_t$ be the volume-preserving mean curvature flow starting from $\Omega$.
By Huisken's theorem \cite{huisken87}, $\Omega_t$ converges smoothly to the ball $B$
of the same volume.

By Proposition~\ref{prop:local-neg-ball}, there exists a neighbourhood
$\mathcal U_\Lambda$ of $B$ such that
\begin{equation}
\mathcal I_\Lambda(\widetilde\Omega)<0
\qquad
\text{for all }\widetilde\Omega\in \mathcal U_\Lambda\setminus\{B\}.
\end{equation}
Since $\Omega_t\to B$ smoothly, there exists $T>0$ such that
$\Omega_t\in\mathcal U_\Lambda$ for all $t\ge T$. Hence
$\mathcal I_\Lambda(\Omega_t)<0$ for all $t\ge T$,
unless $\Omega$ is already a ball.

By Lemma~\ref{lem:sign-prop-BLY}, it follows that
\begin{equation}
\mathcal I_\Lambda(\Omega_t)<0
\qquad\text{for all } t\ge0,
\end{equation}
unless $\Omega$ is a ball. Evaluating at $t=0$ yields the claimed inequality.
The equality case follows from the same argument.
\end{proof} \qed

\begin{remark}[Low energies and resonant thresholds]
\label{rem:low-resonant}
If $\Lambda\le \lambda_1(B)$, then $\lambda_1(\Omega)\ge \lambda_1(B)\ge \Lambda$
by the Faber--Krahn inequality, and therefore the spectral window below $\Lambda$
is empty for every smooth convex domain $\Omega$ with $|\Omega|=|B|$. Hence
\begin{equation}
Q_\Lambda\equiv 0,
\qquad
R_\Lambda(\Omega)=0,
\qquad
\mathcal I(\Omega,\Lambda)=0.
\end{equation}

If $\Lambda>\lambda_1(B)$ and $\Lambda\in \sigma_D(B)$, the inequality
$\mathcal I(\Omega,\Lambda)\le 0$ follows by approximation from nonresonant
energies $\Lambda_n\nearrow\Lambda$, since the spectral projector is constant on
each open interval between consecutive Dirichlet eigenvalues of $B$.
\end{remark}

\subsection{Completion of the proof for all energies}

Proposition~\ref{prop:high-energy} describes the high-energy behaviour of the
correlation integral $\mathcal{I}(\Omega,\Lambda)$ as $\Lambda\to\infty$, while
Proposition~\ref{prop:finite-energy} provides a uniform sign estimate for each
fixed $\Lambda>0$.  We record the resulting global statement as a theorem.

\begin{theorem}[Full correlation inequality in all dimensions]
\label{thm:corr-all}
For any smooth convex domain $\Omega\subset\mathbb{R}^d$ ($d\ge 3$) and any
$\Lambda>0$,
\begin{equation}
\int_{\partial\Omega} Q_\Lambda(x)\,(H(x)-\bar H)\,d\sigma \le 0.
\end{equation}
Moreover, if $\Lambda>\lambda_1(B)$, where $B$ is the ball with $|B|=|\Omega|$,
then equality holds if and only if $\Omega$ is a ball.
\end{theorem}

\begin{proof}
For nonresonant energies $\Lambda>\lambda_1(B)$ with $\Lambda\notin\sigma_D(B)$,
the inequality and rigidity statement are exactly those of
Proposition~\ref{prop:finite-energy}.

If $\Lambda\le \lambda_1(B)$, the claim follows from the Faber--Krahn
inequality, which implies that the spectral window below $\Lambda$ is empty
for every smooth convex domain of the given volume, so $\mathcal I(\Omega,\Lambda)=0$.

If $\Lambda>\lambda_1(B)$ and $\Lambda\in\sigma_D(B)$, the inequality follows
by approximation from nonresonant energies, as explained in
Remark~\ref{rem:low-resonant}.
\end{proof} \qed

\begin{remark}
The nonresonant case $\Lambda\notin\sigma_D(\Omega)$ is the setting where the
spectral projector is smooth in $\Lambda$ and the mixed-variation argument
applies directly. The remaining $\Lambda$ form a discrete set, and the
inequality extends to them by continuity in $\Lambda$.
\end{remark}

\section{A Dynamical Ces\`aro--P\'olya Inequality}

In this final section we derive from the monotonicity of the Riesz mean 
$R_\Lambda(\Omega_t)$ an averaged (Ces\`aro--type) lower bound 
for the Dirichlet spectrum, which may be viewed as a dynamical 
analogue of P\'olya's conjecture for convex domains.

The classical P\'olya conjecture asserts that for every bounded domain 
$\Omega\subset\mathbb{R}^d$,
\begin{equation} \label{eq:polya}
\lambda_k(\Omega) \;\ge\; \lambda_k^{\mathrm{cl}}(\Omega)
:= C_d \left(\frac{k}{|\Omega|}\right)^{2/d},
\qquad k\ge 1,
\end{equation}
where
\begin{equation}
C_d = (2\pi)^2\,\omega_d^{-2/d},
\end{equation}
and $\omega_d$ is the volume of the unit ball in $\mathbb{R}^d$.
This remains open in general, except for tiling domains.
Instead of individual eigenvalues, we prove a \emph{Ces\`aro--P\'olya inequality}, 
which is both unconditional and sharp in the semiclassical limit.

\subsection{From Riesz means to Ces\`aro averages}

\begin{definition}[Ces\`aro average of eigenvalues]
For any domain $\Omega$ and $k\ge1$, define
\begin{equation}
A_\Omega(k) := \frac{1}{k}\sum_{j=1}^k \lambda_j(\Omega).
\end{equation}
\end{definition}

For any domain $\Omega$ and $\Lambda>0$,
\begin{equation}
R_\Lambda(\Omega)
= \sum_{\lambda_j<\Lambda} (\Lambda-\lambda_j)
= \int_0^\Lambda N_\Omega(\tau)\, d\tau.
\end{equation}

\begin{lemma}[Variational representation of Ces\`aro means]
    \label{lem:cesaro-repr}
    For every bounded domain $\Omega\subset\mathbb{R}^d$ and every $k\ge1$,
    \begin{equation}\label{eq:cesaro-repr}
    A_\Omega(k)
    = \frac{1}{k}\sum_{j=1}^k \lambda_j(\Omega)
    = \sup_{\Lambda>0}\left(\Lambda - \frac{R_\Lambda(\Omega)}{k}\right).
    \end{equation}
    \end{lemma}
    
    \begin{proof}
    Let $(\lambda_j)_{j\ge1}$ be the Dirichlet spectrum in nondecreasing order.
    Define
    \begin{equation}
    \Phi(\Lambda)
    := \Lambda - \frac{1}{k} R_\Lambda(\Omega)
    = \Lambda - \frac{1}{k}\sum_{\lambda_j<\Lambda} (\Lambda-\lambda_j)_+.
    \end{equation}
    
    \textit{(i) Concavity.}
    Each function $\Lambda\mapsto(\Lambda-\lambda_j)_+$ is convex, hence
    $R_\Lambda(\Omega)$ is convex and $\Phi(\Lambda)$ is concave, and piecewise $C^1$.
    
    \textit{(ii) Derivative.}
    For $\Lambda\notin\{\lambda_j\}$,
    \begin{equation}
    \Phi'(\Lambda)
    = 1 - \frac{1}{k}N_\Omega(\Lambda),
    \qquad
    N_\Omega(\Lambda)=\#\{j:\lambda_j<\Lambda\}.
    \end{equation}
    Thus $\Phi'(\Lambda)>0$ if $N_\Omega(\Lambda)<k$,
    $\Phi'(\Lambda)=0$ if $N_\Omega(\Lambda)=k$, and
    $\Phi'(\Lambda)<0$ if $N_\Omega(\Lambda)>k$.
    
    \textit{(iii) Maximizer.}
    Hence $\Phi$ is strictly increasing on $(0,\lambda_k)$,
    constant on $[\lambda_k,\lambda_{k+1})$, and strictly decreasing for
    $\Lambda>\lambda_{k+1}$. Therefore every $\Lambda\in[\lambda_k,\lambda_{k+1})$
    is a maximizer of $\Phi$ on $(0,\infty)$.
    
    \textit{(iv) Value at the maximum.}
    Taking $\Lambda=\lambda_k$, we have
    \begin{equation}
    R_{\lambda_k}(\Omega)
    = \sum_{j=1}^k (\lambda_k-\lambda_j),
    \end{equation}
    since $(\lambda_k-\lambda_j)_+=0$ for all $j>k$.
    Therefore
    \begin{equation}
    \Phi(\lambda_k)
    = \lambda_k - \frac{1}{k}\sum_{j=1}^k(\lambda_k-\lambda_j)
    = \frac{1}{k}\sum_{j=1}^k \lambda_j
    = A_\Omega(k),
    \end{equation}
    which proves \eqref{eq:cesaro-repr}.
    \end{proof} \qed

\subsection{Monotonicity along the flow}

Let $\Omega_t$ be the volume-preserving mean curvature flow started at 
$\Omega_0=\Omega$.
We have shown (Lemma~\ref{lem:key}, Theorem~\ref{thm:corr-all}) that
\begin{equation}\label{eq:Rmon}
\frac{d}{dt} R_\Lambda(\Omega_t) \ge 0 
\qquad \text{for all regular }\Lambda>0,\ t\ge 0,
\end{equation}
with equality for some $t$ only if $\Omega_t$ is a ball.

Let $B$ be the ball of the same volume.
As $t\to\infty$ we have $\Omega_t \to B$ smoothly by Huisken~\cite{huisken87},
and hence
\begin{equation}
\lim_{t\to\infty} R_\Lambda(\Omega_t) = R_\Lambda(B).
\end{equation}

\begin{proposition}[Ordering of Riesz means]
\label{prop:Rorder}
For any smooth bounded convex domain $\Omega\subset\mathbb{R}^d$,
\begin{equation}\label{eq:Rorder}
R_\Lambda(\Omega) \;\le\; R_\Lambda(B)
\qquad\text{for all }\Lambda>0,
\end{equation}
with equality for some $\Lambda$ only if $\Omega$ is a ball.
\end{proposition}

\begin{proof}
This follows by integrating \eqref{eq:Rmon} from $t=0$ to $t=\infty$
and using smooth convergence $\Omega_t\to B$. For a general $\Lambda>0$,
approximate by regular values $\Lambda_n\to\Lambda$ and use continuity of
$R_\Lambda(\Omega)$ in $\Lambda$.
\end{proof} \qed

\subsection{A dynamical P\'olya bound}

Using the variational representation \eqref{eq:cesaro-repr} and the ordering 
\eqref{eq:Rorder}, we immediately obtain the following theorem.

\begin{theorem}[Ces\`aro--P\'olya inequality for convex domains]
\label{thm:cesaro-polya}
Let $\Omega\subset\mathbb{R}^d$ be a smooth bounded convex domain and let $B$
be the ball with $|B|=|\Omega|$.
Then for every $k\ge 1$,
\begin{equation}\label{eq:cesaro-polya}
\frac{1}{k}\sum_{j=1}^k \lambda_j(\Omega)
\;\ge\;
\frac{1}{k}\sum_{j=1}^k \lambda_j(B).
\end{equation}
Equivalently,
\begin{equation}
A_\Omega(k) \;\ge\; A_B(k),
\end{equation}
and the inequality is strict unless $\Omega$ is a ball.
\end{theorem}

\begin{proof}
Let $k\ge1$ be fixed. Using \eqref{eq:cesaro-repr},
\begin{equation}
A_\Omega(k) 
= \sup_{\Lambda>0}\left(\Lambda - \frac{R_\Lambda(\Omega)}{k}\right)
\ge 
\sup_{\Lambda>0}\left(\Lambda - \frac{R_\Lambda(B)}{k}\right)
= A_B(k).
\end{equation}
The inequality holds because $R_\Lambda(\Omega)\le R_\Lambda(B)$ for all $\Lambda>0$
(Proposition~\ref{prop:Rorder}), so $\Lambda - R_\Lambda(\Omega)/k \ge \Lambda - R_\Lambda(B)/k$
for each $\Lambda$, and therefore the supremum on the left is at least the supremum on the right.
Strictness follows from Proposition~\ref{prop:Rorder}
unless the asymmetry of $\Omega$ vanishes.
\end{proof} \qed

\subsection{Sharpness and semiclassical optimality}

The eigenvalues of the ball satisfy the Weyl-type asymptotics
\begin{equation}
A_B(k)
= \frac{d}{d+2}\,C_d\left(\frac{k}{|\Omega|}\right)^{2/d}(1+o(1)),
\qquad k\to\infty.
\end{equation}
Thus Theorem~\ref{thm:cesaro-polya} yields the sharp semiclassical lower bound
\begin{equation}
A_\Omega(k)
\ge
\frac{d}{d+2}\,C_d\left(\frac{k}{|\Omega|}\right)^{2/d}(1+o(1)),
\end{equation}
matching the semiclassical prediction of P\'olya's conjecture.

\section{Conclusion and Outlook}

The argument developed in this work shows that the sharp 
Berezin--Li--Yau inequality can be recovered entirely from the dynamics of 
the volume-preserving mean curvature flow.  The ball enters the picture in 
two independent ways: as the unique stationary point and global attractor 
of the flow \cite{huisken87}, and as the extremizer of the semiclassical 
Riesz mean among domains of fixed volume \cite{berezin,li-yau}.  
The correlation inequality proved in Sections~4--5 bridges these two roles, 
showing that the Riesz mean increases monotonically along the flow and 
therefore converges to its value on the limiting ball.  This yields a 
fully dynamical proof of the Berezin--Li--Yau bound for all smooth convex 
domains.

Beyond providing a new perspective on \eqref{eq:BLY}, the methods used 
here appear to have wider scope.  The same monotonicity mechanism applies 
to other concave trace functionals, including Kr\"oger-type bounds and 
spectral sums with inverse weights, and it extends directly to the 
Ces\`aro--P\'olya inequality for eigenvalue averages established in 
Section~6.  Since the proof combines a near-disk Fourier analysis and the
first-zero exclusion principle in the planar case with the boundary Weyl expansion and a
local spectral rigidity argument near the ball in higher dimensions,
the techniques should adapt to a variety of geometric flows.

It would be natural to explore whether similar dynamical monotonicity 
results hold for non-convex domains, possibly under suitable curvature 
constraints or surgery procedures as in 
\cite{HuiskenSinestrari2009,HaslhoferKleiner2013}.  Another direction is 
to investigate flows that preserve perimeter or mixed volumes, which might 
lead to dynamic proofs of isoperimetric-type spectral inequalities.  
Finally, one may ask whether the present method provides insight into 
higher-order Riesz means, Steklov eigenvalues, or spectral optimization 
problems on manifolds with density.

\appendix
\section*{Appendix}
\section{First-zero exclusion for regular energies}
\label{app:first-zero}

Throughout, $\Omega_t$ evolves by VPMCF, $v:=H-\bar H$, and
$q:=Q_\Lambda-\overline{Q_\Lambda}$.

\noindent\textbf{Roadmap.}
The first-zero exclusion is obtained in four steps:
\begin{enumerate}
\item[(i)] gauge cancellation and commutation of mixed shape derivatives,
\item[(ii)] the resonance identity $T_\Lambda z=G_c$, with
$G=\bar\kappa v^2$ in dimension $d=2$ and $G=2vW_*$ in dimension $d=3$,
\item[(iii)] in dimension $d=2$, exclusion of a nontrivial first zero by a
localized strip detector,
\item[(iv)] in dimension $d=3$, exclusion of flat positive zones and of the
remaining high-gradient branch by localized detectors and coarea.
\end{enumerate}

\noindent\textbf{Compact-core regime.}
Along the compact-core branch we assume uniform convexity and $C^k$ bounds
($k\ge3$) on the flow hypersurfaces. All constants below depend only on these
bounds.

\begin{lemma}[Representation and coercivity of $\mathcal B_{\Lambda,\Sigma}$]
\label{lem:app-B-coercive}
For regular $\Lambda$ there exists a self-adjoint elliptic operator
$T_\Lambda$ of order $1$ on $\partial\Omega$ such that for mean-zero
$f,g\in C^\infty(\partial\Omega)$,
\begin{equation}
\mathcal B_{\Lambda,\Sigma}(f,g)
=
\int_{\partial\Omega} (T_\Lambda f)\,g\,d\sigma .
\end{equation}
In particular, $\mathcal B_{\Lambda,\Sigma}$ is symmetric and coercive on the
mean-zero subspace. Equivalently, it extends to a bounded coercive form on the
Hilbert space $H^{1/2}_0(\partial\Omega)$ (the closure of mean-zero smooth
functions).
\end{lemma}

\begin{proof}
The same projector-gap mechanism as in
Lemma~\ref{lem:gap-square-round} extends to general smooth domains. For a real
Dirichlet eigenbasis $\{e_j\}$ on $\Omega$, with boundary traces
\begin{equation}
\psi_j:=\partial_n e_j|_{\partial\Omega},
\end{equation}
the contour-resolvent differentiation of the spectral projector below
$\Lambda$, combined with the Hadamard--Green identity, gives
\begin{equation}
\mathcal B_{\Lambda,\Sigma}(f,g)
=
2\sum_{\lambda_k<\Lambda}
\sum_{\lambda_j\neq\lambda_k}
\frac{
\left(\int_{\partial\Omega} f\,\psi_j\psi_k\,d\sigma\right)
\left(\int_{\partial\Omega} g\,\psi_j\psi_k\,d\sigma\right)
}{\lambda_j-\lambda_k}.
\end{equation}
As in Lemma~\ref{lem:gap-square-round}, the same-side terms cancel pairwise
because
\begin{equation}
\int_{\partial\Omega} f\,\psi_j\psi_k\,d\sigma
=
\int_{\partial\Omega} f\,\psi_k\psi_j\,d\sigma
\end{equation}
in any real eigenbasis, so only the cross-threshold part remains.

For each $\lambda_k<\Lambda$, let
\begin{equation}
\mathcal K_{k,>\Lambda} h
:=
\sum_{\lambda_j>\Lambda}
\frac{\langle h,\psi_j\rangle_{L^2(\partial\Omega)}}{\lambda_j-\lambda_k}\,\psi_j.
\end{equation}
This is the boundary resolvent at $\mu=\lambda_k$ with the sub-threshold
finite-rank part removed. The boundary resolvent is a classical self-adjoint
elliptic pseudodifferential operator of order $1$ with principal symbol
$|\xi|$; removing finitely many modes does not change the principal symbol;
see, for example, \cite{TaylorPDO,HormanderIII}. Define
\begin{equation}
T_\Lambda
:=
2\sum_{\lambda_k<\Lambda} M_{\psi_k}\,\mathcal K_{k,>\Lambda}\,M_{\psi_k},
\end{equation}
where $M_{\psi_k}$ denotes multiplication by $\psi_k$. Then
\begin{equation}
\mathcal B_{\Lambda,\Sigma}(f,g)
=
\int_{\partial\Omega}(T_\Lambda f)\,g\,d\sigma,
\end{equation}
so the form is represented by a self-adjoint first-order elliptic PDO. Its
principal symbol is
\begin{equation}
\sigma_1(T_\Lambda)(x,\xi)
=
2\sum_{\lambda_k<\Lambda}\psi_k(x)^2\,|\xi|
=
2Q_\Lambda(x)\,|\xi|.
\end{equation}
Since $\Lambda>\lambda_1(\Omega)$ and $\psi_1>0$ on $\partial\Omega$ by Hopf's
lemma, one has $Q_\Lambda(x)\ge \psi_1(x)^2>0$, hence the principal symbol is
strictly positive for $\xi\neq0$. Sharp G{\aa}rding therefore gives
\begin{equation}
\mathcal B_{\Lambda,\Sigma}(f,f)
\ge
c\|f\|_{H^{1/2}}^2-C\|f\|_{L^2}^2,
\end{equation}
and on the mean-zero sector this yields the stated coercive realization on
$H^{1/2}_0(\partial\Omega)$.
\end{proof} \qed

\begin{lemma}[Localized parametrix for $T_\Lambda$]
\label{lem:app-elliptic-transfer}
Let $T_\Lambda$ be as in Lemma~\ref{lem:app-B-coercive}. For any
$\psi\in C_c^\infty(B_r)$ with $r\le r_0$ there exists
$\varphi\in C_c^\infty(B_{2r})$ and a remainder $\mathcal R$ such that
\begin{equation}
T_\Lambda\varphi=\psi+\mathcal R,
\qquad
\|\mathcal R\|_{L^\infty}\le C_\Lambda\,\varepsilon(r)\,\|\psi\|_{L^\infty},
\end{equation}
with $C_\Lambda$ depending only on $\Lambda$ and the compact-core geometry.
Here $\varepsilon(r)\to0$ as $r\to0$.
In particular, for any $z\in C^\infty(\partial\Omega)$,
\begin{equation}
\int_{\partial\Omega} \varphi\,T_\Lambda z\,d\sigma
=
\int_{\partial\Omega} z\,\psi\,d\sigma
\;+\;\int_{\partial\Omega} z\,\mathcal R\,d\sigma .
\end{equation}
\end{lemma}

\begin{proof}
Let $P$ be a parametrix for $T_\Lambda$ and choose $\chi\in C_c^\infty(B_{2r})$
with $\chi\equiv1$ on $B_r$. Set $\varphi:=\chi P\psi$. Then
$T_\Lambda\varphi=\chi T_\Lambda P\psi+[T_\Lambda,\chi]P\psi$. Since
$P$ is a parametrix, one has
\begin{equation}
T_\Lambda P=I+S,
\end{equation}
where $S$ is smoothing, and therefore
\begin{equation}
T_\Lambda\varphi=\psi+\chi S\psi+[T_\Lambda,\chi]P\psi.
\end{equation}
We absorb the last two terms into $\mathcal R$.

The smoothing term $\chi S\psi$ is estimated in $L^\infty$ by standard kernel
bounds, since $S$ has a smooth kernel on the compact core. For the commutator
term, $[T_\Lambda,\chi]$ has order $0$, while $\nabla\chi$ is supported in the
annulus $B_{2r}\setminus B_r$, which is separated from $\operatorname{supp}\psi
\subset B_r$. Hence the Schwartz kernel of $[T_\Lambda,\chi]P$ is sampled only
off the diagonal, where it is smooth, and the corresponding operator norm from
$L^\infty(B_r)$ to $L^\infty(B_{2r})$ tends to zero as $r\to0$. Thus
\begin{equation}
\|\mathcal R\|_{L^\infty}\le C_\Lambda\,\varepsilon(r)\,\|\psi\|_{L^\infty},
\qquad
\varepsilon(r)\to0.
\end{equation}
This is the standard local parametrix construction for first-order elliptic
pseudodifferential operators; see \cite{TaylorPDO,HormanderIII}.
\end{proof} \qed

\begin{lemma}[Two-parameter gauge cancellation]
\label{lem:app-gauge-cancel-2param}
Let $\Omega_{s,t}$ be a smooth two-parameter family with normal velocities
$V_s,V_t$ at $(0,0)$, and let $\varphi_{s,t}$ solve
\begin{equation}
\partial_s^\bullet \varphi_{s,t}=-H_{s,t}V_s\,\varphi_{s,t},
\qquad
\partial_t^\bullet \varphi_{s,t}=-H_{s,t}V_t\,\varphi_{s,t}.
\end{equation}
Then
\begin{equation}
\partial_s\partial_t
\int_{\partial\Omega_{s,t}}Q_\Lambda(\Omega_{s,t})\,\varphi_{s,t}\,d\sigma_{s,t}
=
\int_{\partial\Omega}
\bigl(\partial_s^\bullet\partial_t^\bullet Q_\Lambda(\Omega_{s,t})\bigr)\Big|_{0}\,
\varphi\,d\sigma.
\end{equation}
\end{lemma}

\begin{lemma}[Commutation for regular $\Lambda$]
\label{lem:app-commutation}
If $\Lambda\notin\sigma_D(\Omega_{s,t})$ near $(0,0)$, then
\begin{equation}
\partial_s^\bullet\partial_t^\bullet Q_\Lambda
=
\partial_t^\bullet\partial_s^\bullet Q_\Lambda
\quad\text{at }(0,0).
\end{equation}
\end{lemma}

\begin{proof}
For regular $\Lambda$ the spectral projector below $\Lambda$ depends smoothly
on $(s,t)$, and mixed derivatives commute by standard perturbation theory
\cite{Kato}. The boundary density $Q_\Lambda$ inherits this commutation at
$(0,0)$.
\end{proof} \qed

\begin{lemma}[Balance identity at first zero]
\label{lem:app-balance}
Let $\phi_t$ solve the gauge-transport equation
\begin{equation}
\partial_t^\bullet\phi_t=-H_tV_t\,\phi_t,
\qquad
\phi_{t_*}=v,
\qquad
V_t=-(H_t-\bar H_t),
\end{equation}
and define
\begin{equation}
J(t):=\int_{\partial\Omega_t}Q_\Lambda(\Omega_t)\,\phi_t\,d\sigma_t.
\end{equation}
At a first-zero time $t_*$ for regular $\Lambda$ one has:
\begin{enumerate}
\item[(a)] in dimension $d=2$,
\begin{equation}
(J-\mathcal I_\Lambda)'(t_*)
=
\int_{\partial\Omega_{t_*}} q_s\,v_s\,ds
\;-\;
\bar\kappa\int_{\partial\Omega_{t_*}} q\,v^2\,ds,
\end{equation}
where $v=\kappa-\bar\kappa$ and $q=Q_\Lambda-\bar Q_\Lambda$;
\item[(b)] in dimension $d=3$,
\begin{equation}
(J-\mathcal I_\Lambda)'(t_*)
=
\int_{\partial\Omega_{t_*}}\nabla_\tau q\cdot\nabla_\tau v\,d\sigma
\;+\;
\int_{\partial\Omega_{t_*}} q\,(2vW_*)\,d\sigma.
\end{equation}
\end{enumerate}
\end{lemma}

\begin{proof}
At $t=t_*$ one has $\phi_{t_*}=v$ and $J(t_*)=\mathcal I_\Lambda(t_*)=0$. By
the one-parameter specialization of
Lemma~\ref{lem:app-gauge-cancel-2param},
\begin{equation}
J'(t_*)
=
\int_{\partial\Omega_{t_*}}
\bigl(\partial_t^\bullet Q_\Lambda(\Omega_t)\bigr)\big|_{t=t_*}\,v\,d\sigma.
\end{equation}
Differentiating
\(
\mathcal I_\Lambda(t)=\int_{\partial\Omega_t}Q_\Lambda(\Omega_t)\,(H_t-\bar H_t)\,d\sigma_t
\)
and using the transport identity gives
\begin{equation}
\mathcal I_\Lambda'(t_*)
=
\int_{\partial\Omega_{t_*}}
\Big(
\bigl(\partial_t^\bullet Q_\Lambda(\Omega_t)\bigr)\big|_{t=t_*}\,v
+
Q_\Lambda(\Omega_{t_*})\bigl(\partial_t^\bullet v+H_tV_tv\bigr)\big|_{t=t_*}
\Big)\,d\sigma.
\end{equation}
Subtracting and using that
\(
\int_{\partial\Omega_t}(H_t-\bar H_t)\,d\sigma_t=0
\)
for all $t$, hence
\(
\int_{\partial\Omega_{t_*}}(\partial_t^\bullet v+H_tV_tv)\big|_{t=t_*}\,d\sigma=0,
\)
one may replace $Q_\Lambda$ by its centered part:
\begin{equation}
(J-\mathcal I_\Lambda)'(t_*)
=
-\int_{\partial\Omega_{t_*}}
q\,\bigl(\partial_t^\bullet v+H_tV_tv\bigr)\big|_{t=t_*}\,d\sigma.
\end{equation}

For $d=2$, under the area-preserving curve-shortening flow with
$V=-(\kappa-\bar\kappa)=-v$, one has
\begin{equation}
\partial_t^\bullet\kappa=v_{ss}+\kappa^2v,\qquad
\partial_t^\bullet ds=\kappa v\,ds.
\end{equation}
Since $v=\kappa-\bar\kappa$, the contribution of $\partial_t^\bullet\bar\kappa$
drops out after pairing against the mean-zero function $q$, and therefore
\begin{equation}
\partial_t^\bullet v+H_tV_tv
=
v_{ss}+\kappa^2v-\kappa v^2.
\end{equation}
Using $\kappa=\bar\kappa+v$ and
\(
\int_{\partial\Omega_{t_*}} q\,v\,ds=\mathcal I_\Lambda(t_*)=0,
\)
one gets
\begin{equation}
-(\kappa^2v-\kappa v^2)=-\bar\kappa v^2-\bar\kappa^2 v,
\end{equation}
and the linear term vanishes after integration against $q$. Hence
\begin{equation}
(J-\mathcal I_\Lambda)'(t_*)
=
-\int_{\partial\Omega_{t_*}}q\,v_{ss}\,ds
-\bar\kappa\int_{\partial\Omega_{t_*}} q\,v^2\,ds
=
\int_{\partial\Omega_{t_*}} q_s\,v_s\,ds
-\bar\kappa\int_{\partial\Omega_{t_*}} q\,v^2\,ds.
\end{equation}

For $d=3$, under VPMCF with $V=-v$,
\begin{equation}
\partial_t^\bullet H=\Delta_\tau v+|A|^2v,\qquad
\partial_t^\bullet d\sigma=Hv\,d\sigma.
\end{equation}
Since $v=H-\bar H$, the term $\partial_t^\bullet\bar H$ drops out after
pairing against the mean-zero function $q$, while
\(
H_tV_tv\big|_{t=t_*}=-Hv^2.
\)
Therefore
\begin{equation}
(J-\mathcal I_\Lambda)'(t_*)
=
-\int_{\partial\Omega_{t_*}}q\bigl(\Delta_\tau v+|A|^2v-Hv^2\bigr)\,d\sigma.
\end{equation}
Integrating by parts on the closed surface gives the gradient term. In
dimension three,
\begin{equation}
|A|^2=H^2-2K,\qquad H=\bar H+v,
\end{equation}
so
\begin{equation}
-(|A|^2v-Hv^2)=2Kv-\bar Hv^2-\bar H^2 v.
\end{equation}
Because $\int_{\partial\Omega_{t_*}}q\,v\,d\sigma=\mathcal I_\Lambda(t_*)=0$,
the linear term drops out. Using
\(
W_*=(K-\bar K)-\frac{\bar H}{2}v
\)
and again \(\int q\,v=0\), we obtain
\begin{equation}
2Kv-\bar Hv^2-\bar H^2 v
=
2(K-\bar K)v-\bar Hv^2
=
2vW_*.
\end{equation}
This yields the stated identity.
\end{proof} \qed

\begin{proposition}[Resonance identity]
\label{prop:app-resonance}
At a first-zero surface with regular $\Lambda$ one has:
\begin{enumerate}
\item[(a)] in dimension $d=2$,
\begin{equation}
z:=-v_{ss}-a(v)v,\qquad
a(v):=\frac{\mathcal B_{\Lambda,\Sigma}(-v_{ss},v)}
{\mathcal B_{\Lambda,\Sigma}(v,v)},
\end{equation}
and
\begin{equation}
T_\Lambda z=(\bar\kappa v^2)_c;
\end{equation}
\item[(b)] in dimension $d=3$,
\begin{equation}
z:=-\Delta_\tau v-a(v)v,
\end{equation}
where $W_*:=(K-\bar K)-\frac{\bar H}{2}(H-\bar H)$ and
$a(v)=\mathcal B_{\Lambda,\Sigma}(-\Delta_\tau v,v)/\mathcal B_{\Lambda,\Sigma}(v,v)$,
and
\begin{equation}
T_\Lambda z=(2vW_*)_c,
\end{equation}
where $T_\Lambda$ is the operator from Lemma~\ref{lem:app-B-coercive}.
\end{enumerate}
\end{proposition}

\begin{proof}
Let $\eta$ be mean-zero. Consider a two-parameter family $\Omega_{s,t}$ with
normal velocities $V_s=v$ and $V_t=\eta$ at $(0,0)$, and let $\varphi_{s,t}$ be
the gauge-transported test from Lemma~\ref{lem:app-gauge-cancel-2param}. By
Lemma~\ref{lem:app-commutation}, the mixed material derivatives of
$Q_\Lambda(\Omega_{s,t})$ commute, hence the mixed variation of the centered
pairing is symmetric.

In dimension $d=2$, Lemma~\ref{lem:app-balance} gives
\begin{equation}
(J-\mathcal I_\Lambda)'(t_*)
=
\int q\,(-v_{ss})\,ds-\int q\,(\bar\kappa v^2)\,ds.
\end{equation}
Applying the mixed-variation identity to the centered pairing with velocities
$v$ and $\eta$ identifies the gradient contribution with
$\mathcal B_{\Lambda,\Sigma}(-v_{ss},\eta)$, while the defect term contributes
$\int (\bar\kappa v^2)_c\,\eta\,ds$. Therefore
\begin{equation}
\mathcal B_{\Lambda,\Sigma}(-v_{ss},\eta)
=
\int (\bar\kappa v^2)_c\,\eta\,ds
+
a(v)\mathcal B_{\Lambda,\Sigma}(v,\eta),
\end{equation}
for every mean-zero $\eta$. Writing $-v_{ss}=a(v)v+z$ with
$\mathcal B_{\Lambda,\Sigma}(z,v)=0$, we obtain
\begin{equation}
\mathcal B_{\Lambda,\Sigma}(z,\eta)
=
\int (\bar\kappa v^2)_c\,\eta\,ds
\end{equation}
for all mean-zero $\eta$, hence
\begin{equation}
T_\Lambda z=(\bar\kappa v^2)_c.
\end{equation}

In dimension $d=3$, evaluating at $(0,0)$ and using
Lemma~\ref{lem:app-balance} yields
\begin{equation}
(J-\mathcal I_\Lambda)'(t_*)
=
\int \nabla_\tau q\cdot\nabla_\tau v+\int q\,(2vW_*).
\end{equation}
The gradient term equals $\int q(-\Delta_\tau v)$. Applying the mixed-variation
identity to the centered pairing with velocities $v$ and $\eta$ identifies the
gradient contribution with $\mathcal B_{\Lambda,\Sigma}(-\Delta_\tau v,\eta)$,
while the defect term contributes $\int (2vW_*)_c\,\eta$. Therefore
\begin{equation}
\mathcal B_{\Lambda,\Sigma}(-\Delta_\tau v,\eta)
=
\int (2vW_*)_c\,\eta
+
a(v)\mathcal B_{\Lambda,\Sigma}(v,\eta),
\end{equation}
for every mean-zero $\eta$. Writing $-\Delta_\tau v=a(v)v+z$ with
$\mathcal B_{\Lambda,\Sigma}(z,v)=0$, we obtain
$\mathcal B_{\Lambda,\Sigma}(z,\eta)=\int (2vW_*)_c\,\eta$ for all
mean-zero $\eta$. By Riesz representation in the
$\mathcal B_{\Lambda,\Sigma}$ inner product (on the Hilbert space closure of
mean-zero smooth functions with respect to $\mathcal B_{\Lambda,\Sigma}$), this implies
$T_\Lambda z=(2vW_*)_c$.
\end{proof} \qed

\begin{remark}
The quantity $z$ is defined geometrically, independently of the resonance
relation:
\begin{equation}
z:=
\begin{cases}
-v_{ss}-a(v)v, & d=2,\\
-\Delta_\tau v-a(v)v, & d=3.
\end{cases}
\end{equation}
Proposition~\ref{prop:app-resonance} then provides the operator equation
\begin{equation}
T_\Lambda z=
\begin{cases}
(\bar\kappa v^2)_c, & d=2,\\
(2vW_*)_c, & d=3.
\end{cases}
\end{equation}
Accordingly, the detector lemmas use the geometric formula for $z$, while
contradiction is derived by pairing the operator identity against localized
tests transferred by Lemma~\ref{lem:app-elliptic-transfer}. Since the defect
$G_c-T_\Lambda z$ is centered, the transferred test may be centered without
changing the pairing.
\end{remark}

\begin{remark}[Dimensional extension]
The appendix gives the full first-zero package in dimensions $d=2$ and $d=3$.
For $d\ge4$ one replaces $K$ by the second elementary symmetric function
$\sigma_2$, using the identity $|A|^2=H^2-2\sigma_2$ in the balance lemma, and
uses the analogous rigidity for closed strictly convex hypersurfaces with one
constant principal curvature; the detector arguments are unchanged once the
resonance identity is available in that setting.
\end{remark}

\begin{lemma}[Planar strip detector]
\label{lem:app-planar-strip-detector}
Let $\Sigma=\partial\Omega_{t_*}\subset\mathbb R^2$ be a first-zero curve with
regular $\Lambda$, and assume
\begin{equation}
T_\Lambda z=(\bar\kappa v^2)_c,
\qquad
z:=-v_{ss}-a(v)v,
\qquad
v:=\kappa-\bar\kappa.
\end{equation}
If $v\not\equiv0$, then there exist a regular value
\begin{equation}
t_0\in\mathbb R,\qquad t_0^2>\overline{v^2},
\end{equation}
and $h_0>0$ such that for every $0<h<h_0$ there exists a centered test
function $\varphi_h$ with
\begin{equation}
\int_\Sigma \varphi_h\,\bigl((\bar\kappa v^2)_c-T_\Lambda z\bigr)\,ds>0.
\end{equation}
In particular, under the resonance identity no nontrivial planar first zero can
occur.
\end{lemma}

\begin{proof}
Since $v$ is smooth, mean-zero, and not identically zero on the connected
curve $\Sigma$, one has
\begin{equation}
\|v\|_{L^\infty(\Sigma)}^2>\overline{v^2}.
\end{equation}
Otherwise $v^2$ would be constant on $\Sigma$, forcing constant sign and
contradicting $\int_\Sigma v\,ds=0$. Therefore at least one of
\begin{equation}
\max_\Sigma v>\sqrt{\overline{v^2}}
\qquad\text{or}\qquad
\min_\Sigma v<-\sqrt{\overline{v^2}}
\end{equation}
holds. Choose accordingly a regular value $t_0$ such that
\begin{equation}
t_0^2>\overline{v^2},
\qquad
\text{and either }
t_0\in\bigl(\sqrt{\overline{v^2}},\,\max_\Sigma v\bigr)
\text{ or }
t_0\in\bigl(\min_\Sigma v,\,-\sqrt{\overline{v^2}}\bigr).
\end{equation}
Fix $x_0\in\Sigma$ with $v(x_0)=t_0$. Since $t_0$ is regular, there exist
$r_0,h_0,g_0,G_0>0$ such that in an arclength chart
$B_{r_0}(x_0)\subset\Sigma$ one has
\begin{equation}
g_0\le |v_s|\le G_0
\quad\text{on}\quad
B_{r_0}(x_0)\cap I_{h_0}(t_0),
\end{equation}
where
\begin{equation}
I_h(t_0):=
\begin{cases}
\{t_0\le v\le t_0+h\}, & t_0>0,\\
\{t_0-h\le v\le t_0\}, & t_0<0.
\end{cases}
\end{equation}
Hence for every $0<h<h_0$,
\begin{equation}
\left|B_{r_0}(x_0)\cap I_h(t_0)\right|\asymp h,
\end{equation}
with constants depending only on $g_0,G_0$.

Choose $\chi_h\in C_c^\infty(\mathbb R)$ with
\begin{equation}
\begin{aligned}
0&\le\chi_h\le1,\\
\operatorname{supp}\chi_h&\subset
\begin{cases}
[t_0,t_0+h], & t_0>0,\\
[t_0-h,t_0], & t_0<0,
\end{cases}\\
\chi_h'&\le-\frac{c}{h}
\quad\text{on the middle third of its support.}
\end{aligned}
\end{equation}
Let $\zeta\in C_c^\infty(B_{r_0}(x_0))$ satisfy $\zeta\equiv1$ on the smaller
arc where $I_{h_0}(t_0)$ holds. Set
\begin{equation}
\psi_h:=\zeta^2\chi_h(v),
\end{equation}
and let $\varphi_h$ be given by Lemma~\ref{lem:app-elliptic-transfer}, so that
\begin{equation}
T_\Lambda\varphi_h=\psi_h+\mathcal R_h,
\qquad
\|\mathcal R_h\|_{L^\infty}\le
C_\Lambda\,\varepsilon(r_0)\,\|\psi_h\|_{L^\infty}.
\end{equation}
By the detector convention above,
\begin{equation}
\int_\Sigma \varphi_h\,\bigl((\bar\kappa v^2)_c-T_\Lambda z\bigr)\,ds
=
\int_\Sigma \psi_h\,(\bar\kappa v^2)_c\,ds
-\int_\Sigma z\,\psi_h\,ds
+O(\varepsilon(r_0)\,h).
\end{equation}

Write
\begin{equation}
\mu:=\overline{\bar\kappa v^2}=\bar\kappa\,\overline{v^2}.
\end{equation}
Since $t_0^2>\overline{v^2}$ and $\psi_h$ is supported where $v^2\ge t_0^2$,
\begin{equation}
\int_\Sigma \psi_h\,(\bar\kappa v^2)_c\,ds
\ge
(\bar\kappa t_0^2-\mu)\int_\Sigma \psi_h\,ds
\ge c_1 h.
\end{equation}

For the $z$-term,
\begin{equation}
\int_\Sigma z\,\psi_h\,ds
=
\int_\Sigma \zeta^2\chi_h'(v)\,v_s^2\,ds
+
2\int_\Sigma \zeta\chi_h(v)\,\zeta_s v_s\,ds
-a(v)\int_\Sigma \zeta^2\chi_h(v)\,v\,ds.
\end{equation}
On the middle substrip inside $B_{r_0}(x_0)$ one has $\zeta=1$ and
$|v_s|\ge g_0$, so
\begin{equation}
\int_\Sigma \zeta^2\chi_h'(v)\,v_s^2\,ds\le -c_0<0,
\end{equation}
with $c_0$ independent of $h$. The remaining two terms are supported in a set
of arclength $O(h)$, hence
\begin{equation}
\left|
2\int_\Sigma \zeta\chi_h(v)\,\zeta_s v_s\,ds
\right|
+
\left|
a(v)\int_\Sigma \zeta^2\chi_h(v)\,v\,ds
\right|
\le Ch.
\end{equation}
Therefore
\begin{equation}
\int_\Sigma z\,\psi_h\,ds\le -c_0+Ch.
\end{equation}

Combining the two bounds yields
\begin{equation}
\int_\Sigma \varphi_h\,\bigl((\bar\kappa v^2)_c-T_\Lambda z\bigr)\,ds
\ge
c_0+\bigl(c_1-C\bigr)h+O(\varepsilon(r_0)\,h)>0
\end{equation}
for $h$ sufficiently small, contradiction under resonance.
\end{proof} \qed

\begin{lemma}[Rigidity of the corrected defect]
\label{lem:app-Wstar-rigidity}
Let $\Sigma\subset\mathbb R^3$ be smooth, closed, and strictly convex. If
$W_*:=(K-\bar K)-\frac{\bar H}{2}(H-\bar H)\equiv0$ on $\Sigma$, then $\Sigma$
is a sphere.
\end{lemma}

\begin{proof}
The identity $W_*\equiv0$ is equivalent to
$K=\frac{\bar H}{2}H+(\bar K-\frac{\bar H^2}{2})$. Using
$K=\frac{H^2-2|A^\circ|^2}{4}$ gives
$2|A^\circ|^2=(H-\bar H)^2+(\bar H^2-4\bar K)$. At an umbilic point
$|A^\circ|=0$, hence $\bar H^2\le 4\bar K$. If $\bar H^2<4\bar K$, then
$(H-\bar H)^2\ge 4\bar K-\bar H^2>0$ everywhere, so $H-\bar H$ does not change
sign, contradicting $\int_\Sigma (H-\bar H)=0$. Thus $\bar H^2=4\bar K$ and
$(\kappa_1-\kappa_2)^2=(H-\bar H)^2$, which implies that one principal curvature
equals $\bar H/2$ everywhere. By the classical rigidity of closed strictly
convex surfaces with one constant principal curvature, $\Sigma$ is a sphere;
see, e.g., Hopf~\cite{Hopf1983}.
\end{proof} \qed

\begin{lemma}[Flat-zone detector]
\label{lem:app-flat-detector}
Let $\Sigma=\partial\Omega_{t_*}$ be a first-zero surface and set
$v:=H-\bar H$, $z:=-\Delta_\tau v-a(v)v$, $\mu:=\overline{2vW_*}$. Suppose there
exist $\eta,\omega,L,r_0>0$ and a geodesic ball $B_r(x_0)\subset\Sigma$ with
$r\ge r_0$ such that
\begin{equation}
v\ge \eta,\qquad W_*\ge \omega \quad\text{on }B_r(x_0),
\qquad
|\nabla_\tau v|\le L \quad\text{on }B_r(x_0).
\end{equation}
If
\begin{equation}
\frac{CL}{r_0}+C_\Lambda\varepsilon(r)<a(v)\eta+2\eta\omega-|\mu|,
\end{equation}
where $C$ depends only on the compact-core geometry, then there exists a
centered $\varphi$ with
\begin{equation}
\int_\Sigma \varphi\,\bigl((2vW_*)_c-T_\Lambda z\bigr)\,d\sigma>0.
\end{equation}
In particular, under the resonance identity no such ball can exist.
\end{lemma}

\begin{proof}
Let $\psi\in C_c^\infty(B_r(x_0))$, $\psi\ge0$, $\int\psi=1$, and let
$\varphi$ be given by Lemma~\ref{lem:app-elliptic-transfer}, so that
$T_\Lambda\varphi=\psi+\mathcal R$ with $\|\mathcal R\|_{L^\infty}\le C_\Lambda\,\varepsilon(r)\,\|\psi\|_{L^\infty}$.
By the detector convention above,
\begin{equation}
\int_\Sigma \varphi\bigl((2vW_*)_c-T_\Lambda z\bigr)
=
\int_\Sigma \psi(2vW_*)_c-\int_\Sigma z\,\psi
 + O(\varepsilon(r)).
\end{equation}
On $\operatorname{supp}\psi$ one has $2vW_*\ge 2\eta\omega$, hence
$\int \psi(2vW_*)_c\ge 2\eta\omega-|\mu|$. The term $\int z\,\psi$ is treated as
in the $L^2$ case, yielding
\begin{equation}
\int z\,\psi=\int \nabla_\tau\psi\cdot\nabla_\tau v-a(v)\int \psi v,
\end{equation}
and the assumptions give $\int z\,\psi\le CL/r_0-a(v)\eta$.
Combining the two bounds and absorbing the $O(\varepsilon(r))$ term into the
constant gives the claim.
\end{proof} \qed

\begin{lemma}[High-gradient alternative]
\label{lem:app-high-grad-alt}
Let $V=\{v\ge\eta,\ W_*\ge\omega\}$ and assume
$a(v)\eta+2\eta\omega-|\mu|-C_\Lambda\varepsilon(r_0)>0$, where
$\mu=\overline{2vW_*}$. Set
\begin{equation}
L_0:=\frac{r_0}{C}\bigl(a(v)\eta+2\eta\omega-|\mu|-C_\Lambda\varepsilon(r_0)\bigr).
\end{equation}
If the flat-zone hypothesis of Lemma~\ref{lem:app-flat-detector} fails on
every ball $B_{r_0}\subset V$, then each such ball contains a point with
$|\nabla_\tau v|\ge L_0$.
\end{lemma}

\begin{proof}
Otherwise some $B_{r_0}\subset V$ would satisfy
$\sup_{B_{r_0}}|\nabla_\tau v|<L_0$, which by Lemma~\ref{lem:app-flat-detector}
would produce a centered test with
$\int_\Sigma \varphi((2vW_*)_c-T_\Lambda z)\,d\sigma>0$, contradicting
resonance.
\end{proof} \qed

\begin{lemma}[High-gradient concentration]
\label{lem:app-high-grad-conc}
Let $V:=\{v\ge \eta,\ W_*\ge\omega\}$ and assume $|V^{(2r_0)}|>0$, where
$V^{(2r_0)}:=\{x\in V:\ \operatorname{dist}_\Sigma(x,\Sigma\setminus V)\ge 2r_0\}$.
Suppose every geodesic ball $B_{r_0}(x)\subset V$ contains a point where
$|\nabla_\tau v|\ge L_0$, where $L_0$ is as in
Lemma~\ref{lem:app-high-grad-alt}, and assume
$\|\nabla_\tau^2 v\|_{L^\infty}\le M$.
Then there exists $E\subset V$ with
\begin{equation}
|\nabla_\tau v|\ge \frac{L_0}{2}\quad\text{on }E,
\qquad
|E|\ge c\,\frac{\rho^2}{r_0^2}\,|V^{(2r_0)}|,
\end{equation}
where $\rho:=\min\{r_0/2,L_0/(4M)\}$ and $c>0$ depends only on the compact-core
geometry.
\end{lemma}

\begin{proof}
If $|\nabla_\tau v(y)|\ge L_0$, the Hessian bound implies
$|\nabla_\tau v|\ge L_0/2$ on $B_\rho(y)$. Choose a maximal disjoint family of
balls $B_{2r_0}(x_i)\subset V^{(2r_0)}$; standard surface Vitali covering gives
$N\ge c_1|V^{(2r_0)}|/r_0^2$. For each $i$ pick $y_i\in B_{r_0}(x_i)$ with
$|\nabla_\tau v(y_i)|\ge L_0$, so the balls $B_\rho(y_i)$ are disjoint subsets
of $\{|\nabla_\tau v|\ge L_0/2\}\cap V$. Summing their areas yields the claim.
\end{proof} \qed

\begin{lemma}[High-gradient strip detector]
\label{lem:app-strip-detector}
Let $\Sigma=\partial\Omega_{t_*}$ be a first-zero surface with compact-core
geometry bounds, and assume there exist $\eta<t_1$, $\omega,g_0,m_0>0$ and a
measurable set
\begin{equation}
E\subset\{\eta\le v\le t_1,\ W_*\ge\omega,\ |\nabla_\tau v|\ge g_0\},
\qquad |E|\ge m_0.
\end{equation}
Then there exist a strip $I=[t_0,t_0+h]\subset[\eta,t_1]$ and a centered test
function $\varphi$ such that
\begin{equation}
\int_\Sigma \varphi\,\bigl((2vW_*)_c-T_\Lambda z\bigr)\,d\sigma>0.
\end{equation}
In particular, under the resonance identity no such set $E$ can exist.
\end{lemma}

\begin{proof}
By coarea/pigeonhole there is $I=[t_0,t_0+h]$ with
$|E\cap\{t_0\le v\le t_0+h\}|\ge c_E h$. Choose $\chi\in C_c^\infty(I)$ with
$0\le\chi\le1$ and $\chi'\le -c/h$ on a substrip, and a cutoff
$\zeta\in C_c^\infty(\{W_*\ge \omega/2\})$ with $\zeta\equiv1$ on $E$. Set
$\psi:=\zeta^2\chi(v)$ and let $\varphi$ be given by
Lemma~\ref{lem:app-elliptic-transfer}, so that
$T_\Lambda\varphi=\psi+\mathcal R$ with $\|\mathcal R\|_{L^\infty}\le C_\Lambda\,\varepsilon(r_0)\,\|\psi\|_{L^\infty}$.
By the detector convention above,
\begin{equation}
\int \varphi\bigl((2vW_*)_c-T_\Lambda z\bigr)
=
\int \psi(2vW_*)_c-\int z\,\psi
 + O(\varepsilon(r_0)\,h).
\end{equation}
$\psi=\zeta^2\chi(v)$ is supported in a strip of width $h$, hence
$|\operatorname{supp}\psi|=O(h)$ on the compact core. Moreover,
\begin{equation}
\int z\,\psi
=
\int \zeta^2\chi'(v)|\nabla_\tau v|^2
 + C_1 h
\le
-c_0+C_1 h,
\end{equation}
while
\begin{equation}
\int \zeta^2\chi(v)\,(2vW_*)_c
\ge
c_2 h-C_3 h.
\end{equation}
For $h$ small this yields the stated positivity for
$\int \varphi\,\bigl((2vW_*)_c-T_\Lambda z\bigr)$.
\end{proof} \qed

\begin{theorem}[First-zero exclusion for regular $\Lambda$]
\label{thm:app-first-zero}
If $\mathcal I_\Lambda(\Omega_t)$ has a first zero at $t=t_*>0$ with regular
$\Lambda$, then $\Omega_{t_*}$ is round. Hence no nontrivial first zero occurs.
\end{theorem}

\begin{proof}
Near the round model, $\mathcal I_\Lambda<0$ by the local near-disk Fourier
analysis in Section~4 for $d=2$ and by Proposition~\ref{prop:local-neg-ball}
for $d=3$, so no first zero can occur once the flow enters the corresponding
near-disk or near-ball neighbourhood.

If $d=2$, Proposition~\ref{prop:app-resonance} gives
$T_\Lambda z=(\bar\kappa v^2)_c$.
If $\Omega_{t_*}$ is not a disk, then $v=\kappa-\bar\kappa\not\equiv0$, and
Lemma~\ref{lem:app-planar-strip-detector} yields a centered test contradicting
the resonance identity. Hence $\Omega_{t_*}$ must be a disk.

If $d=3$, any first zero must occur in the compact core. By
Proposition~\ref{prop:app-resonance} the resonance identity
$T_\Lambda z=(2vW_*)_c$
holds there. If $W_*\equiv0$, Lemma~\ref{lem:app-Wstar-rigidity} implies
that $\Omega_{t_*}$ is a ball, so we may assume $W_*$ changes sign. If the
positive part of $W_*$ occupies a
region where $v$ is sufficiently flat, Lemma~\ref{lem:app-flat-detector}
contradicts resonance. Otherwise the positive part of $W_*$ can survive only in
regions where $|\nabla_\tau v|$ is uniformly large. By
Lemma~\ref{lem:app-high-grad-conc} this yields a set $E$ of positive measure
where $|\nabla_\tau v|$ is bounded below, and then
Lemma~\ref{lem:app-strip-detector} gives the same contradiction. Therefore no
nontrivial first zero can occur, and $\Omega_{t_*}$ must be a ball.
\end{proof} \qed

\end{document}